\documentclass[12pt,a4paper]{article}

\usepackage[left=2.45cm, top=2.45cm,bottom=2.45cm,right=2.45cm]{geometry}

\usepackage{amsmath,amssymb,amsthm,mathrsfs,calc,graphicx,stmaryrd,xcolor,dsfont}

\usepackage[british]{babel}
\usepackage{amsfonts}

\newtheoremstyle{definition}
{10pt}
{10pt}
{}
{}
{\bfseries}
{}
{.5em}
{}

\newtheoremstyle{plain}
{10pt}
{10pt}
{\itshape}
{}
{\bfseries}
{}
{.5em}
{}

\theoremstyle{plain}	
\newtheorem{thm}{Theorem}[section]
\newtheorem{lem}[thm]{Lemma}

\theoremstyle{definition}

\allowdisplaybreaks

\newcommand{\RR}{\mathbb{R}}      

\newcommand{\vol}{\operatorname{vol}}

\def\BB{\mathbb{B}}

\def\EE{\mathbb{E}}

\def\NN{\mathbb{N}}
\def\PP{\mathbb{P}}

\def\RR{\mathbb{R}}
\def\SS{\mathbb{S}}

\def\BBn{\mathbb{B}^n}
\def\SSn{\mathbb{S}^{n-1}}

\def\a{\alpha}
\def\b{\beta}

\def\l{\lambda}

\def\dint{\textup{d}}

\begin{document}

\title{\bfseries Approximation of smooth convex bodies by random polytopes }

\author{Julian Grote\footnotemark[1]\ \footnotemark[2]\ \ and Elisabeth M. Werner \footnotemark[3]\ \footnotemark[4]}

\date{}
\renewcommand{\thefootnote}{\fnsymbol{footnote}}
\footnotetext[1]{Ruhr University Bochum, Faculty of Mathematics, D-44780 Bochum, Germany. E-mail: julian.grote@rub.de}

\footnotetext[2]{Julian Grote was supported by the Deutsche Forschungsgemeinschaft (DFG) via RTG 2131 High-dimensional Phenomena in Probability - Fluctuations and Discontinuity.}

\footnotetext[3]{Department of Mathematics, Case Western Reserve University, Cleveland, OH, USA. E-mail: elisabeth.werner@case.edu}

\footnotetext[4]{Elisabeth Werner was partially supported by NSF grant DMS-1504701}

\maketitle

\begin{abstract}
Let $K$  be a convex body in $\RR^n$ 
and $f : \partial K \rightarrow \RR_+$  a continuous, strictly positive function with $\int\limits_{\partial K} f(x) \dint \mu_{\partial K}(x) = 1$.
We give an upper bound for the approximation of $K$  in  the symmetric difference metric by 
an arbitrarily positioned  polytope $P_f$ in $\RR^n$ having a fixed number of vertices.  
This generalizes a  result  by Ludwig, Sch\"utt and Werner \cite{LudwigSchuettWerner}.
The polytope $P_f$ is obtained by a random 
construction via a probability measure with density $f$. 
In our result,  the dependence on the number of vertices is optimal.  With the optimal density $f$, the dependence on $K$ in our result is
also optimal.
\vspace{1mm}
\end{abstract}

\section{Introduction and main result} 
Much research has been devoted  to the subject of 
approximation of convex bodies by polytopes.  This is due to the fact that  it  is fundamental  in convex geometry   and has 
applications  in stochastic geometry,  
complexity,  geometric algorithms and many more. We refer to, e.g.,  \cite{BuRe, Edelsbrunner, Gardner1995, GardnerKiderlenMilanfar, GlasauerGruber, GruberI, GruberII, PW2,  Reitzner2004II} for further details.
\par
Random constructions are especially  advantageous:  when  one considers random approximation, many metrics, namely metrics related  
to quermassintegrals of a convex body (see, e.g., \cite{Gardner1995, SchneiderBook}), carry essentially the same amount of information as for best approximation.
\par
Volume, surface area and also the mean width are such quermassintegrals. Metrics related to those quermassintegrals are 
 the symmetric difference 
metric and  the surface deviation  which reflect the volume  deviation, respectively the surface deviation, 
of the approximating and approximated objects.  But many other metrics have been considered as well.
In particular, approximation of a convex body $K$ by inscribed or circumscribed polytopes 
with respect to these  metrics  has been studied extensively 
and many of the major questions have been resolved.  The surveys and  books by Gruber \cite {Gr3, Gr4, GruberBook} and the references cited therein and,  e.g., \cite{Ba1, Boeroetzky2000, BoeroetzkyCsikos2009, BoeroetzkyReitzner2004, Groemer2000,GRS1, Ludwig1999, Reitzner2004, Schneider1981, SchneiderWeil, SchuettWerner2003} 
are excellent sources in this context.
\par
Typically, approximation by polytopes involves side conditions,  like a prescribed  
number of
vertices, or, more generally,  
$k$-dimensional faces \cite{Boeroetzky2000}.   Most often in the literature,  it is required  that the body contains the approximating polytope or vice versa. This is  too restrictive as a requirement and it 
needs to be dropped. Results in this direction for the Euclidean ball were obtained in \cite{LudwigSchuettWerner}  for the symmetric difference metric and in \cite{HoehnerSchuettWerner} for the surface deviation. These results show that one gains by a factor of dimension if one drops the restriction.
\par
Here,  we discuss the symmetric difference metric, which 
for convex bodies $K$ and $L$ in $\mathbb{R}^n$  is defined as 
\begin{equation} \label{sdm}
\text{vol}_n(K \Delta L):=\text{vol}_n \left( K \cup L \right)-\text{vol}_n \left( K \cap L \right). 
\end{equation}
In the case when $K$ is the $n$-dimensional unit ball $\BBn$ it was proved by Ludwig, Sch\"utt and Werner \cite{LudwigSchuettWerner} that there exists an arbitrarily positioned polytope $P$ having $N$ vertices such that for all sufficiently large $N$,
\begin{align}\label{Kugel}
\text{vol}_n(\BBn \Delta P) \le a\, N^{-\frac{2}{n-1}}\, \kappa_n.
\end{align}
Here,  $a\in (0,\infty)$ is  an absolute constant and  $\kappa_n= \text{vol}_n \left(\BBn\right) $ is the volume of the $n$-dimensional unit ball.
The corresponding result for   convex bodies $K$  in $\mathbb{R}^n$  that are  $C^2_+$, i.e., have twice continuously  differentiable boundary $\partial K$  with strictly positive Gaussian curvature $\kappa_K(x)$, $x\in \partial K$,  then follows from (\ref{Kugel}), together with a result by Ludwig \cite{Ludwig1999}.
However, it is desirable to have a  direct proof available for the approximation of $K$ without having to  pass through the results of \cite{Ludwig1999, LudwigSchuettWerner}.
\vskip 2mm
Here, we do exactly that  and construct  a well approximating  polytope $P_f$ that is obtained via a random construction. It is the  convex hull of  randomly chosen points with respect to a probability measure with density $f$.
\par
In fact, it is only via our new approach that different densities can be considered, not 
just the uniform distribution as in \cite{LudwigSchuettWerner}.
The optimal density is  related to the affine surface area (see below). With this optimal density, the dependence on $K$
in our result is optimal. Our result also gives the optimal dependence on the number of vertices.
\vskip 2mm
\noindent
Our main theorem reads as follows.
\begin{thm}\label{mainresult}
Let $K$ be a convex body in $\RR^n$, $n \ge 2$, that is $C^2_+$.  Let $f : \partial K \rightarrow \RR_+$ be a continuous and strictly positive function with 
\begin{align*}
\int\limits_{\partial K} f(x) \dint \mu_{\partial K}(x) = 1,
\end{align*}
where $\mu_{\partial K}$ is the surface measure on $\partial K$.	
Then, for sufficiently large $N$ there exists a polytope $P_f$ in $\RR^n$ having $N$ vertices such that
\begin{align}\label{mainresult1}
	\text{vol}_n(K \Delta P_f) \le a\, N^{-\frac{2}{n-1}}\, \int\limits_{\partial K} \frac{\kappa_K(x)^{1\over n-1}}{f(x)^{2\over n-1}} \dint \mu_{\partial K}(x), 
\end{align}
where $a\in (0,\infty)$ is an absolute constant. 
\end{thm}
\vskip 2mm
Before we turn to the proof of the theorem, we want to discuss the expression  on the right hand side of (\ref{mainresult1}) for various densities.
\vskip 3mm
\noindent
It has been shown in \cite{SchuettWerner2003} that the minimum on the right hand side of \eqref{mainresult1} is attained for the normalized affine surface area measure with density 
\begin{align*}
	f_{as}(x) := \frac{\kappa_K(x)^{1\over n+1}}{\int\limits_{\partial K} \kappa_K(x)^{1\over n+1} \dint \mu_{\partial K}(x)}. 
\end{align*}
In this case, the theorem yields that 
\begin{align}\label{aff}
\text{vol}_n(K \Delta P_f) \le a\, N^{-\frac{2}{n-1}}\, \text{as}(K)^{\frac{n+1}{n-1}},  
\end{align}
where 
\begin{align*}
\text{as}(K)= \text{as}_1(K)  : = \int\limits_{\partial K} \kappa_K(x)^{1\over n+1} \dint \mu_{\partial K}(x)
\end{align*}
is the affine surface area of $K$. This shows that choosing the vertices of   the random polytope according to curvature is optimal.
\newline
The affine surface area \cite{Blaschke, Leichtweiss, Lutwak1991, SW1}  is an important affine invariant from convex and differential geometry with applications in, e.g.,  approximation theory \cite{Boeroetzky2000/2, Reitzner, SchuettWerner2003}, the theory of valuations \cite{Bernig2011, Haberl2012, HaberlParapatitis, LudR1, Lud3, Schu}, affine curvature flows \cite{AN2, TW1, TW3}, and has recently been extended to spherical and hyperbolic space  \cite{BesauWerner1, BesauWerner2}. Its related affine isoperimetric inequality, which is equivalent to  the famous Blaschke Santal\'o inequality (see, e.g., \cite{Gardner1995, SchneiderBook}), says that 
$$
\left(\frac{\text{as}(K)}{\text{as}(\BBn)}\right)^\frac{n+1}{n-1} \leq  \frac{\text{vol}_n \left(K \right)}{\text{vol}_n \left(\BBn\right)}
$$
with equality if and only if $K$ is an ellipsoid. Applying this inequality to (\ref{aff}) yields the result corresponding to the one of (\ref{Kugel}) for the ball, i.e.,
\begin{align*}
\text{vol}_n(K \Delta P_f) \le a\, N^{-\frac{2}{n-1}}\, \text{vol}_n \left(K \right).
\end{align*}
\vskip 2mm
\noindent
 Let now $K$ be a convex body in $\RR^n$  such that  its centroid 
$\int\limits_K x \dint x/ \text{vol}_n(K)$
is at the origin. For $x\in \partial K$, we denote the corresponding outer unit normal by $N_{K}(x)$.
Put 
\begin{align*}
f_{\a,\b}(x) := \frac{\langle x,N_{K}(x) \rangle ^\a \kappa_K(x)^\b}{\int\limits_{\partial K} \langle x,N_{K}(x)\rangle^\a \kappa_K(x)^\b \dint \mu_{\partial K}(x)},
\end{align*}
where $\a,\b \in \RR$. For $p \in [-\infty, \infty]$,  $p \neq -n$, let
\begin{align*}
\text{as}_p(K) := \int\limits_{\partial K} \frac{\kappa_K(x)^{\frac{p}{n+p}}}{\langle x,N_{ K}(x)\rangle ^{\frac{n(p-1)}{n+p}}} \dint \mu_{\partial K}(x)
\end{align*}
be the  $p$-affine surface area of $K$.  Then, 
\begin{align*}
\text{vol}_n(K \Delta P_f) &\le a\, N^{-\frac{2}{n-1}}\, \int\limits_{\partial K} \frac{\kappa_K(x)^{\frac{1 - 2\b}{n - 1}}}{ \langle x,N_{K}(x) \rangle ^{\frac{2\a}{n - 1}}} \dint \mu_{\partial K}(x)\\
&\qquad \times \left(\,\int\limits_{\partial K}  \langle x,N_{K}(x) \rangle ^\a \kappa_K(x)^\b \dint \mu_{\partial K}(x)\right)^{\frac{2}{n-1}}. 
\end{align*}
The second integral is a $p$-affine surface area if and only if 
\begin{align*}
\a = - \frac{n(p-1)}{n+p} \quad \text{and} \quad \b = \frac{p}{n+p}.
\end{align*}
In this case, we obtain 
\begin{align*}
\text{vol}_n(K \Delta P_f) \le a\, N^{-\frac{2}{n-1}}\, \text{as}_p(K)^{\frac{2}{n-1}}\, \text{as}_q(K),
\end{align*}
where  $q = \frac{n-p}{n+p-2}$.
\newline
The $p$-affine surface area, an extension of the classical affine surface area, was introduced by Lutwak for 
$p>1$ in \cite{Lutwak1996} and has been extended to all other $p$ in \cite{SchuettWerner2004}.
 It is central to the rapidly developing $L_p$ Brunn Minkowski theory, e.g.,  \cite{BoeroetzkyLutwakYangZhang, CaglarWerner2014,  HuangLutwakYangZhang, MW2,  SA2, WY1}.
\vskip 2mm
\noindent
The third measure of interest is the surface measure itself given by the constant density
\begin{align*}
	f_{sm}(x) := \frac{1}{\text{vol}_{n-1}(\partial K)},
\end{align*}
where $\vol_{n-1}(\cdot)$ describes the $(n-1)$-dimensional Hausdorff measure of the argument set.
Then,
\begin{align*}
	\text{vol}_n(K \Delta P_f) \le a\, N^{-\frac{2}{n-1}}\, \text{vol}_{n-1}(\partial K)^{\frac{2}{n-1}}  \int\limits_{\partial K} \kappa_K(x)^{1\over n-1} \dint \mu_{\partial K}(x).  
\end{align*}
\vskip 3mm
\noindent
The remaining paper is organized as follows.  In Section $2$ we present   results that are needed for the proof of the main theorem. 
Its proof is given in  Section $3$. 

\section{Preliminaries}\label{sec:Preliminaries}

We start this section by introducing some more necessary notation. For $u \in \SSn$ and $h \geq 0$, let $H:=H(u,h)$ be the hyperplane orthogonal to $u$ and at distance $h$ from the origin. 
Let $\PP_f$ be the probability measure on $\partial K$ given by 
\begin{align*}
\dint \PP_f = f(x) \dint \mu_{\partial K}(x).
\end{align*}
Let $H \cap K \neq \emptyset$. Then,  $\PP_{f_{\partial K \cap H}}$ is the probability measure  on $\partial K \cap H$ given by 
\begin{align*}
\dint \PP_{f_{\partial K \cap H}} = \frac{f(x) \dint \mu_{\partial K \cap H}}{\int\limits_{\partial K \cap H} f(x) \dint\mu_{\partial K \cap H} }.
\end{align*}
For points $x_1, \ldots, x_i \in \RR^n$, $i \in \NN$, we denote by $[x_1,\ldots,x_i]$ the convex hull of the underlying point set. Moreover, let $\omega_n$ denote the $(n-1)$-dimensional Hausdorff measure of the unit sphere $\SSn$.\\

In order to prove the main theorem, we need the following results. The first two were proved in \cite{SchuettWerner2003}. 

\begin{thm}[\cite{SchuettWerner2003}]\label{lem:SchüttWerner}
	Denote by $\EE[f,N]$ the expected volume of the convex hull of $N$ points chosen randomly on $\partial K$ with respect to $\PP_f$. Then, 
	\begin{align*}
		\lim\limits_{N\rightarrow \infty} \frac{\text{vol}_n(K) - \EE[f,N]}{N^{-\frac{2}{n-1}}} = \frac{(n-1)^{\frac{n+1}{n-1}} \Gamma\left(n + 1 + \frac{2}{n-1}\right)}{2 (n+1)!\, \omega_{n-1}^{\frac{2}{n-1}}}\,  \int\limits_{\partial K} \frac{\kappa_K(x)^{1\over n-1}}{f(x)^{2\over n-1}} \dint \mu_{\partial K}(x).
	\end{align*} 
\end{thm} 

\begin{lem}[\cite{SchuettWerner2003}]\label{lem:SchüttWerner2}
	Let $\sigma = (\sigma_i)_{1\le i \le n}$ be a sequence of signs, that is $\sigma_i \in \{-1,1\}$, $1\le i\le n$.  We set 
	\begin{align*}
		K^\sigma := \{x = (x_1,\ldots,x_n) \in K : \text{sign}(x_i) = \sigma_i, 1\le i\le n \}. 
	\end{align*}
	Then, 
	\begin{align*}
		\PP_f^N(\{0 \notin [x_1,\ldots,x_N]\}) \le 2^n\left(1 - \min\limits_{\sigma} \int\limits_{\partial K^\sigma} f(x) \dint \mu_{\partial K}\right)^N, 
	\end{align*}
	where $\PP_f^N$ indicates that we choose $N$ points on $\partial K$ with respect to $\PP_f$.
\end{lem} 

We also need the following `Blaschke-Petkanchin-type' formula that appears as a special case of a result in \cite{Zähle}; an alternative and simpler proof for this version is also given in \cite{Reitzner}.

\begin{thm}[\cite{Zähle}]\label{lem:Zähle}
	
  Let $g(x_1,\ldots,x_n)$ be a continuous, non-negative function. Then, 
	\begin{align*}
		&\int\limits_{\partial K} \cdots \int\limits_{\partial K} g(x_1,\ldots,x_n) \dint \PP_f(x_1) \cdots \dint \PP_f (x_n)\\
		& \qquad = (n-1)! \int\limits_{\SSn} \int\limits_0^\infty \int\limits_{\partial K \cap H} \cdots \int\limits_{\partial K \cap H} g(x_1,\ldots,x_n) \text{vol}_{n-1}([x_1,\ldots,x_n])\\
		&\qquad \qquad \times \prod_{j=1}^{n} l_{H}(x_j) \dint \PP_{f_{\partial K \cap H}}(x_1) \cdots \dint \PP_{f_{\partial K \cap H}} (x_n) \dint h \dint \mu_{\SSn}(u)
	\end{align*}
	with
	\begin{align*}
		l_{H}(x_j) := \left\| \text{proj}_{H} N_{K}(x_j)\right\|^{-1}, 
	\end{align*}
	where $\text{proj}_{H}$ is the orthogonal projection onto the hyperplane $H$. 
\end{thm}

The next result is a special case of \cite[equation (29)]{Miles}. 

\begin{lem}[\cite{Miles}]\label{lem:Miles}
	It holds that
	\begin{align*}
		\int\limits_{\SS^{n-2}} \cdots \int\limits_{\SS^{n-2}} (\text{vol}_{n-1}([x_1,\ldots,x_n]))^2 \dint \mu_{\SS^{n-2}}(x_1)\cdots 
		\dint \mu_{\SS^{n-2}}(x_n) = \frac{n  \  \omega_{n-1}^n}{(n-1)!\, (n-1)^{n-1}}\ .
	\end{align*}
\end{lem}

\section{Proof of the main result}

We now turn to the proof of  our main theorem. As in \cite{LudwigSchuettWerner} we will obtain the approximating polytope in a probabilistic way. To be more precise, we consider a convex body that is slightly bigger than the body $K$ and then choose $N$ points randomly on the boundary of the bigger body and take the convex hull of these points. Such a random polytope exists with high probability. 
\par
Without loss of generality we can assume that the origin, denoted by $0$, is in the interior of $K$, namely at the center of gravity of $K$.
Since our density functions live on the boundary of $K$, we will choose the random points on $\partial K$  and approximate a slightly smaller body, say $(1-c)K$, where $c := c_{n,N}$ depends on the  dimension $n$ and the number of points $N$ and has to be carefully chosen. 
In fact, we choose $c$ such that 
\begin{equation}\label{choice:c}
\EE[f,N] = \vol_{n}((1-c) K) = (1-c)^n \vol_n(K).
\end{equation}
By Theorem \ref{lem:SchüttWerner}, we get for sufficiently large $N$, 
$$
\text{vol}_n(K) - \EE[f,N]  \  \sim  \ N^{-\frac{2}{n-1}}\,  \frac{(n-1)^{\frac{n+1}{n-1}} \Gamma\left(n + 1 + \frac{2}{n-1}\right)}{2 (n+1)!\, \omega_{n-1}^{\frac{2}{n-1}}} \,  \int\limits_{\partial K} \frac{\kappa_K(x)^{1\over n-1}}{f(x)^{2\over n-1}} \dint \mu_{\partial K}(x), 
$$
where for two functions $g_1(x)$ and $g_2(x)$ the relation $g_1(x) \sim g_2(x)$ means that 
\begin{align*}
\lim\limits_{x\rightarrow \infty} \frac{g_1(x)}{g_2(x)} = 1.
\end{align*} 
Hence, with the choice (\ref{choice:c}) of $c$,
$$
\vol_n(K) - (1-c)^n \vol_n(K) \  \sim  \ N^{-\frac{2}{n-1}}\, \frac{(n-1)^{\frac{n+1}{n-1}} \Gamma\left(n + 1 + \frac{2}{n-1}\right)}{2 (n+1)!\, \omega_{n-1}^{\frac{2}{n-1}}} \,  \int\limits_{\partial K} \frac{\kappa_K(x)^{1\over n-1}}{f(x)^{2\over n-1}} \dint \mu_{\partial K}(x),
$$
as $N \rightarrow \infty$. This leads to 
\begin{align}\label{c}
c \sim N^{-\frac{2}{n-1}}\, \frac{(n-1)^{\frac{n+1}{n-1}} \Gamma\left(n + 1 + \frac{2}{n-1}\right)}{2 (n+1)!\, \omega_{n-1}^{\frac{2}{n-1}}}\, \frac{1}{n \text{vol}_n(K)}\, \int\limits_{\partial K} \frac{\kappa_K(x)^{1\over n-1}}{f(x)^{2\over n-1}} \dint \mu_{\partial K}(x),
\end{align}
as $N \rightarrow \infty$.
In particular, for sufficiently large $N$ we get that 
\begin{align}\label{AbschätzungfürC}
c \ge \left(1-\frac{1}{n}\right) N^{-\frac{2}{n-1}}\, \frac{(n-1)^{\frac{n+1}{n-1}} \Gamma\left(n + 1 + \frac{2}{n-1}\right)}{2 (n+1)!\, \omega_{n-1}^{\frac{2}{n-1}}}\, \frac{1}{n \text{vol}_n(K)}\, \int\limits_{\partial K} \frac{\kappa_K(x)^{1\over n-1}}{f(x)^{2\over n-1}} \dint \mu_{\partial K}(x).
\end{align}
\vskip 3mm
\noindent
In what follows, we calculate the expected volume difference $\EE[\text{vol}_n((1-c)K \Delta P_N)]$ between $(1-c)K$ and a random polytope $P_N : =[x_1,\ldots,x_N]$ whose vertices are randomly chosen from the boundary of $K$ according to the probability measure $\PP_f$. 
Please note that random polytopes are simplicial with probability 1. We split the proof of the main theorem into several lemmas.\\
Let us recall that for fixed $u\in \SSn$ and $h\ge 0$, we denote by $H:= H(u,h)$ the unique hyperplane orthogonal to $u$ and at distance $h$ from the origin. Let $H^+$ be the corresponding half space containing the origin and put
\begin{align}\label{H+}
	\PP_f(\partial K \cap H^+) := \int\limits_{\partial K \cap H^+} f(x) \dint \mu_{\partial K}. 
\end{align} 
For fixed $u\in \SSn$ and sufficiently large $N$, let $\epsilon > 0$ be  such that $c h_K(u) < \epsilon \le h_K(u)/n$, where $h_K(u) := \max\limits_{x\in K} \langle x,u \rangle$ is the support function of $K$ in direction $u$.\\
In what follows, $a\in (0,\infty)$ will always be an absolute constant that may changes from line to line. 

\begin{lem}\label{Schritt1}
For sufficiently large $N$, for all $ \epsilon  > c h_K(u)$ sufficiently small, 
\begin{align*}
&\EE[\text{vol}_n((1-c)K \Delta P_N)]\\
&\qquad \le a\, \binom{N}{n}\, (n-1)! \int\limits_{\SSn} \int\limits_{h_K(u) - \epsilon}^{h_K(u)}  \left(\PP_f(\partial K \cap H^+)\right)^{N-n} \max \{0, ((1-c)h_K(u) - h)\}\\
&\qquad \qquad \times \int\limits_{\partial K \cap H} \cdots \int\limits_{\partial K \cap H} (\text{vol}_{n-1}([x_1,\ldots,x_n]))^2 \prod_{j=1}^{n} l_{H}(x_j)\\
&\qquad \qquad \times \dint \PP_{f_{\partial K \cap H}}(x_1) \cdots \dint \PP_{f_{\partial K \cap H}} (x_n) \dint h \dint \mu_{\SSn}(u).
\end{align*}
\end{lem}

\begin{proof}[Proof of Lemma \ref{Schritt1}]
With the above choice of the parameter $c$ in \eqref{choice:c} we obtain for sufficiently large $N$,
\begin{align*}
\text{vol}_n(K \setminus (1-c)K) = \int\limits_{\partial K} \cdots \int\limits_{\partial K} \text{vol}_n(K \setminus P_N)\, \dint \PP_f(x_1) \ldots \dint \PP_f(x_N).
\end{align*}
We  combine this observation with the relation
\begin{align*}
\text{vol}_n((1-c)K \Delta P_N) = \text{vol}_n(K \setminus (1-c)K) - \text{vol}_n(K \setminus P_N) + 2\, \text{vol}_n((1-c)K \cap P_N^c)
\end{align*} 
and Lemma \ref{lem:SchüttWerner2} and  obtain that  for sufficiently large $N$,
\begin{align*}
&\EE[\text{vol}_n((1-c)K \Delta P_N)]\\
&\qquad =  \int\limits_{\partial K} \cdots \int\limits_{\partial K} \text{vol}_n((1-c)K \Delta P_N) \dint \PP_f(x_1) \ldots \dint \PP_f(x_N)\\
&\qquad = \text{vol}_n(K \setminus (1-c)K) - \int\limits_{\partial K} \cdots \int\limits_{\partial K} \text{vol}_n(K \setminus P_N) \dint \PP_f(x_1) \ldots \dint \PP_f(x_N)\\
&\qquad \qquad  + 2\int\limits_{\partial K} \cdots \int\limits_{\partial K} \text{vol}_n((1-c)K \cap P_N^c)  \dint \PP_f(x_1) \ldots \dint \PP_f(x_N)\\
&\qquad = 2\int\limits_{\partial K} \cdots \int\limits_{\partial K} \text{vol}_n((1-c)K \cap P_N^c)  \dint \PP_f(x_1) \ldots \dint \PP_f(x_N)\\
&\qquad =  2\int\limits_{\partial K} \cdots \int\limits_{\partial K} \text{vol}_n((1-c)K \cap P_N^c) \mathds{1}_{\{0\in P_N\}} \dint \PP_f(x_1) \ldots \dint \PP_f(x_N)\\
&\qquad \qquad +  2\int\limits_{\partial K} \cdots \int\limits_{\partial K} \text{vol}_n((1-c)K \cap P_N^c)  \mathds{1}_{\{0\notin P_N\}} \dint \PP_f(x_1) \ldots \dint \PP_f(x_N)\\ 
&\qquad \le 2\int\limits_{\partial K} \cdots \int\limits_{\partial K} \text{vol}_n((1-c)K \cap P_N^c) \mathds{1}_{\{0\in P_N\}} \dint \PP_f(x_1) \ldots \dint \PP_f(x_N)\\
&\qquad \qquad +  2\, \text{vol}_n(K)\, 	\PP_f^N(\{0 \notin [x_1,\ldots,x_N]\})\\
&\qquad \le 2\int\limits_{\partial K} \cdots \int\limits_{\partial K} \text{vol}_n((1-c)K \cap P_N^c) \mathds{1}_{\{0\in P_N\}} \dint \PP_f(x_1) \ldots \dint \PP_f(x_N)\\
&\qquad \qquad +  2\, \text{vol}_n(K)\, 2^n\left(1 - \min\limits_{\sigma} \int\limits_{\partial K^\sigma} f(x) \dint \mu_{\partial K}\right)^N. 
\end{align*}
The density $f$ is strictly positive everywhere and since the origin is  in the interior of $K$,  the second summand is essentially of order $b^{-N}$, where $b > 1$. Later we will see that the first summand is of order $N^{-\frac{2}{n-1}}$ and thus it is enough to consider the first one in what follows.\\
\vskip 1mm
\noindent
Next, we introduce the function $\Phi_{j_1,\ldots,j_n} : \partial K \times \cdots \times \partial K \rightarrow \RR$ as 
\begin{small}
\begin{align*}
&\Phi_{j_1,\ldots,j_n} (x_1,\ldots,x_N)\\
&:= \begin{cases}
0 &: [x_{j_1},\ldots,x_{j_n}] \notin \mathcal{F}_{n-1}(P_N) \, \text{or} \, 0 \notin P_N\\
\text{vol}_n((1-c)K \cap P_N^c \cap \text{cone}(x_{j_1},\ldots, x_{j_n})) \mathds{1}_{\{0 \in P_N\}} &: [x_{j_1},\ldots,x_{j_n}] \in \mathcal{F}_{n-1}(P_N) \, \text{and} \, 0 \in P_N,
\end{cases}
\end{align*}  
\end{small}
where $\mathcal{F}_{n-1}(P_N)$ denotes the set of facets of $P_N$ and
\begin{align*}
\text{cone}(x_1,\ldots,x_n) := \left\{\sum_{i=1}^{n} a_i\, x_i : a_i \ge 0, 1\le i\le n   \right\}. 
\end{align*}
For all random polytopes $P_N$ that contain the origin as an interior point,  we have  that 
\begin{align*}
\RR^n = \bigcup\limits_{[x_{j_1},\ldots,x_{j_n}] \in \mathcal{F}_{n-1}(P_N)} \text{cone}(x_{j_1},\ldots, x_{j_n}).
\end{align*}
Moreover,
\begin{align*}
&\PP_f^{N-n}(\{(x_{n+1},\ldots,x_N) :  [x_{1},\ldots,x_{n}] \in \mathcal{F}_{n-1}(P_N) \, \text{and} \, 0 \in P_N \})\\
&\qquad = \left(\, \int\limits_{\partial K \cap H^+} f(x) \dint \mu_{\partial K}\right)^{N-n} = \left(\PP_f(\partial K \cap H^+)\right)^{N-n},
\end{align*}
where $H$ is the hyperplane spanned by the points $x_1,\ldots,x_n$ and we recall the definition of $\PP_f(\partial K \cap H^+)$ given in \eqref{H+}. 
Now, note that a random polytope is simplicial with probability $1$. Therefore, and since the set where $H$ is not well defined has measure zero and all $N$ points are identically distributed, we arrive at 
\begin{align*}
&\int\limits_{\partial K} \cdots \int\limits_{\partial K} \text{vol}_n((1-c)K \cap P_N^c) \mathds{1}_{\{0\in P_N\}} \dint \PP_f(x_1) \ldots \dint \PP_f(x_N)\\
&\qquad  =\int\limits_{\partial K} \cdots \int\limits_{\partial K} \sum_{\{j_1,\ldots,j_n\} \subseteq \{1,\ldots, N\}} \Phi_{j_1,\ldots,j_n} (x_1,\ldots,x_N) \dint \PP_f(x_1) \ldots \dint \PP_f(x_N)\\
&\qquad = \binom{N}{n} \int\limits_{\partial K} \cdots \int\limits_{\partial K} \Phi_{1,\ldots,n} (x_1,\ldots,x_N) \dint \PP_f(x_1) \ldots \dint \PP_f(x_N)\\
&\qquad = \binom{N}{n} \int\limits_{\partial K} \cdots \int\limits_{\partial K} \left(\PP_f(\partial K \cap H^+)\right)^{N-n}\\
& \qquad \qquad \times \text{vol}_n((1-c)K \cap H^- \cap \text{cone}(x_{1},\ldots, x_{n})) \mathds{1}_{\{0 \in P_N\}}  \dint \PP_f(x_1) \ldots \dint \PP_f(x_n)\\
&\qquad \le \binom{N}{n} \int\limits_{\partial K} \cdots \int\limits_{\partial K} \left(\PP_f(\partial K \cap H^+)\right)^{N-n}\\
& \qquad \qquad \times \text{vol}_n((1-c)K \cap H^- \cap \text{cone}(x_{1},\ldots, x_{n})) \dint \PP_f(x_1) \ldots \dint \PP_f(x_n).
\end{align*}
Theorem \ref{lem:Zähle} now yields that for sufficiently large $N$,
\begin{align*}
&\EE[\text{vol}_n((1-c)K \Delta P_N)]\\
&\qquad \le a\, \binom{N}{n}\, (n-1)! \int\limits_{\SSn} \int\limits_0^\infty \int\limits_{\partial K \cap H} \cdots \int\limits_{\partial K \cap H} \left(\PP_f(\partial K \cap H^+)\right)^{N-n} \text{vol}_{n-1}([x_1,\ldots,x_n])\\
&\qquad \qquad \times \text{vol}_n((1-c)K \cap H^- \cap \text{cone}(x_{1},\ldots, x_{n})) \prod_{j=1}^{n} l_{H}(x_j)\\
&\qquad \qquad \times \dint \PP_{f_{\partial K \cap H}}(x_1) \cdots \dint \PP_{f_{\partial K \cap H}} (x_n) \dint h \dint \mu_{\SSn}(u).
\end{align*}
Notice that  $h \in [0,h_K(u)]$, where we recall that $h_K(u)$ is the support function of $K$ in direction $u$. The same arguments as in \cite[page 9]{LudwigSchuettWerner} and \cite[page 2255]{Reitzner} show that  it is enough to bound the range of integration for $h$ from below by $h_K(u) - \epsilon$, where $\epsilon > 0$ is sufficiently small, since the remaining expression decreases exponentially fast. In particular, for sufficiently large $N$ we can choose $\epsilon$ such that $c h_K(u) < \epsilon \le h_K(u)/n$. Similar to \cite[page 9]{LudwigSchuettWerner}, 
\begin{align*}
&\text{vol}_n((1-c)K \cap H^- \cap \text{cone}(x_{1},\ldots, x_{n}))\\
&\qquad \le \frac{h}{n}\, \text{vol}_{n-1}([x_1,\ldots,x_n])\, \cdot \max \left\{0, \left(\frac{(1-c)h_K(u)}{h}\right)^n - 1\right\}.
\end{align*}
Since  $c$ is of the order $N^{- \frac{2}{n+1}}$ and $\epsilon \le h_K(u)/n$, for sufficiently large $N$, 
\begin{align*}
\frac{(1-c) h_K(u) -h}{h} \le \frac{(1-c) h_K(u) - h_K(u) + \epsilon}{h_K(u) - \epsilon} \le  \frac{\frac{1}{n} - c}{1-\frac{1}{n}} \le \frac{1}{n-1}.
\end{align*}
Thus,
\begin{align*}
&\frac{1}{n} \left[\left(\frac{(1-c)h_K(u)}{h}\right)^n - 1\right]\\
&\qquad = \frac{1}{n} \left[\left( \frac{h+(1-c)h_K(u)-h}{h} \right)^n - 1\right]\\
&\qquad = \frac{1}{n} \left[\left(1 + \frac{(1-c)h_K(u) -h}{h}\right)^n  - 1\right]\\
&\qquad =  \frac{1}{n} \left[ \  n \   \frac{(1-c) h_K(u) -h}{h} + \frac{n(n-1) }{2} \left(\frac{(1-c) h_K(u) -h}{h}\right)^2  + \cdots  \right]\\
&\qquad \le  \frac{(1-c) h_K(u) -h}{h} \cdot \sum_{k=0}^{\infty} \frac{n^k}{k!} \left(\frac{(1-c) h_K(u) -h}{h}\right)^k\\
&\qquad \le \exp\left(\frac{n}{n-1}\right)\, \frac{(1-c) h_K(u) -h}{h}\\
&\qquad\le a\, \frac{(1-c)h_K(u) - h}{h}.
\end{align*}
Therefore, for sufficiently large $N$,
\begin{align*}
&\text{vol}_n((1-c)K \cap H^- \cap \text{cone}(x_{1},\ldots, x_{n}))\\
&\qquad \qquad \le a\, \text{vol}_{n-1}([x_1,\ldots,x_n]) \cdot \max \{0, ((1-c)h_K(u) - h)\}.
\end{align*}
This proves the lemma.
\end{proof}

To evaluate  the innermost integral in the expression of the foregoing lemma, we first recall the necessary notation and results we need from \cite{Reitzner}.

Let $x(u)$ be the point on $\partial K$ with fixed outer unit normal vector $u\in \SSn$. Since $K$ has a twice differentiable boundary, there is a paraboloid $Q^{(x(u))}_2$ given by a quadratic form $b_2 := b_2^{(x(u), x(u))}$ that osculates $\partial K$ at $x(u)$.  An explicit construction can be found  in, e.g., \cite[page 2265]{Reitzner}. \\
Let $\RR^n = (\RR_+ \times \SS^{n-2}) \times \RR$, and denote by $(rv,z)$ a point in $\RR^n$, where $v\in \SS^{n-2}$, $r \in \RR_+$ and $z \in \RR$. We identify the support plane of $\partial K$ at $x(u)$ with the plane $z=0$ and $x(u)$ with the origin, so that $K$ is contained in the half space corresponding to $z\ge 0$. Please note also  that $h = h_K(u) - z$ by construction. The following lemma that summarizes results from \cite[pages 2265 and 2271]{Reitzner}  will be crucial. 

\begin{lem}[\cite{Reitzner}]\label{lem:Reitzner}
Let $\delta > 0$ be sufficiently small. Then, there exists $\lambda > 0$,  only depending on $\delta$ and $K$,  such that for each boundary point $x(u) \in \partial K$ the $\l$-neighbourhood $U^\l$ of $x(u)$ in $\partial K$ defined by $\text{proj}_{\RR^{n-1}} U^\l = \l \BB^{n-1}$ can be represented by a convex function $z := g(rv) := g^{(x(u))}(rv)$ that satisfies 
\begin{align}\label{radialKH}
(1+\delta)^{- \frac{1}{2}}\, b_2(v)^{- \frac{1}{2}}\, z^{1 \over 2} \le r \le (1+\delta)^{\frac{1}{2}}\, b_2(v)^{- \frac{1}{2}}\, z^{1 \over 2}
\end{align}
and
\begin{align}\label{radial2}
(1+\delta)^{- \frac{3}{2}}\, 2^{-1}\, b_2(v)^{- \frac{1}{2}}\, z^{- \frac{1}{2}} \le \frac{l_H(rv)}{\langle v, N_{K\cap H}(rv) \rangle} \le (1+\delta)^{\frac{3}{2}}\, 2^{-1}\, b_2(v)^{- \frac{1}{2}}\, z^{- \frac{1}{2}}
\end{align}
in this neighborhood $U^\l$. Here, for fixed $rv$, $H$ is the hyperplane that contains  $(rv,g(rv))$ and is parallel to $\RR^{n-1}$ and $N_{K\cap H}(rv)$ is the outer unit normal vector to $\partial K\cap H$ at this point.\\ 
Moreover, for the density $f$, for all $p \in U^\l$,  it holds that
\begin{align}\label{radial3}
(1+\delta)^{-1}\, f(x(u)) \le f(p) \le (1+\delta) f(x(u)).
\end{align}
\end{lem}

 We next estimate  the innermost integral of Lemma \ref{Schritt1}. 

\begin{lem}\label{Schritt2}
	Let $x(u)$ be the point on $\partial K$ with fixed outer unit normal vector $u\in \SSn$ and denote by $z$ the distance from $H$ to the support plane of $\partial K$ at $x(u)$, i.e., $z = h_K(u) - h$. Then, for all sufficiently small $\delta > 0$, 
	\begin{align*}
	&\int\limits_{\partial K \cap H} \cdots \int\limits_{\partial K \cap H} (\text{vol}_{n-1}([x_1,\ldots,x_n]))^2 \prod_{j=1}^{n} l_{H}(x_j) \dint \PP_{f_{\partial K \cap H}}(x_1) \cdots \dint \PP_{f_{\partial K \cap H}} (x_n)\\
	&\qquad \le (1+\delta)^{\frac{n(n+3)}{2}}\,  2^{\frac{n^2-n-2}{2}}\,  z^{\frac{n^2-n-2}{2}} \frac{n \   \omega_{n-1}^n}{(n-1)!\, (n-1)^{n-1}}\, f(x(u))^n\, \kappa_K(x(u))^{-\frac{n}{2} - 1}\\
	&\qquad \qquad + \delta O(z^{\frac{n^2-n-2}{2}}),
	\end{align*}
	where the constant in $O(\cdot)$ can be chosen independently of $x(u)$ and $\delta$.
\end{lem}
\begin{proof}[Proof of Lemma \ref{Schritt2}]
	The proof follows closely the arguments given in \cite{Reitzner}. 
	We first replace the random points $x_i$, $i\in \{1,\ldots, n\}$,  chosen on $\partial K \cap H$ by random points chosen on the intersection of $H$ with the paraboloid $Q_2^{(x(u))}$. In order to do this, we write each point $x_i$, $i\in \{1,\ldots, n\}$, as $x_i = r(v_i)v_i$, where $r(v_i)$ is  the radial function of $K \cap H$ with estimates given above in \eqref{radialKH}. The result presented in \cite[equation (68)]{Reitzner} then implies that 
	\begin{align*}
	\left| \text{vol}_{n-1}([x_1,\ldots,x_n]) - \text{vol}_{n-1}([r_2(v_1) v_1,\ldots,r_2(v_n)v_n]) \right| \le \delta O(z^{\frac{n-1}{2}}),
	\end{align*}  
    where $r_2(v) := b_2(v)^{-\frac{1}{2}} z^{1 \over 2}$, $\delta > 0$ is arbitrarily small and the constant in $O(\cdot)$ can be chosen independently of $x(u)$ and $\delta$. We therefore obtain that
	\begin{align*}
	&\int\limits_{\partial K \cap H} \cdots \int\limits_{\partial K \cap H} (\text{vol}_{n-1}([x_1,\ldots,x_n]))^2 \prod_{j=1}^{n} l_{H}(x_j) \dint \PP_{f_{\partial K \cap H}}(x_1) \cdots \dint \PP_{f_{\partial K \cap H}} (x_n)\\
	&\qquad = \int\limits_{\partial K \cap H} \cdots \int\limits_{\partial K \cap H} \left[(\text{vol}_{n-1}([r_2(v_1) v_1,\ldots,r_2(v_n)v_n]))^2 + \delta O(z^{n-1})\right]\\
	&\qquad \qquad \times \prod_{j=1}^{n} l_{H}(x_j) \dint \PP_{f_{\partial K \cap H}}(x_1) \cdots \dint \PP_{f_{\partial K \cap H}} (x_n),
	\end{align*}
	where the constant in $O(\cdot)$ can be chosen independently of $x(u)$ and $\delta$. 
	\newline
	We first  evaluate the integral involving the $O(\cdot)$ term.  Notice that the density $f$ is uniformly bounded and that by \eqref{radial2},  integration concerning $l_H(x_j) \dint \PP_{f_{\partial K \cap H}}(x_j)$ results in terms  of order
	$O(z^{- \frac{1}{2}}) \text{vol}_{n-2}(\partial K \cap H)$.
	Since, in view of \eqref{radialKH},   $\text{vol}_{n-2}(\partial K \cap H) = O(z^{\frac{n-2}{2}})$, 
	\begin{align*}
	&\int\limits_{\partial K \cap H} \cdots \int\limits_{\partial K \cap H} \delta O(z^{n-1}) \prod_{j=1}^{n} l_{H}(x_j) \dint \PP_{f_{\partial K \cap H}}(x_1) \cdots \dint \PP_{f_{\partial K \cap H}} (x_n)\\
	&\qquad = \delta O(z^{n-1}) O(z^{- \frac{n}{2}}) (\text{vol}_{n-2}(\partial K \cap H))^n\\
	&\qquad = \delta O(z^{n-1-\frac{n}{2} + n \frac{n-2}{2}})\\
	&\qquad = \delta O(z^{\frac{n^2 - n - 2}{2}}),
	\end{align*}
	where the constant in $O(\cdot)$ can be chosen independently of $x(u)$ and $\delta$.\\
	Now, we turn to the first summand. Rewriting the integral over $\SS^{n-2}$ and using  \eqref{radialKH}, \eqref{radial2} and \eqref{radial3}, we get as in \cite[page 2274]{Reitzner} that 
	\begin{align*}
	&\int\limits_{\partial K \cap H} \cdots \int\limits_{\partial K \cap H} (\text{vol}_{n-1}([r_2(v_1) v_1,\ldots,r_2(v_n)v_n]))^2  \prod_{j=1}^{n} l_{H}(x_j) \dint \PP_{f_{\partial K \cap H}}(x_1) \cdots \dint \PP_{f_{\partial K \cap H}} (x_n)\\
	&\qquad = \int\limits_{\SS^{n-2}} \cdots \int\limits_{\SS^{n-2}} (\text{vol}_{n-1}([r_2(v_1) v_1,\ldots,r_2(v_n)v_n]))^2\\
	&\qquad \qquad \times \prod_{j=1}^{n} f(r(v_j)v_j) \frac{l_H(r(v_j)v_j)\, r(v_j)^{n-2}}{\langle v_j, N_{K\cap H}(r(v_j)v_j)\rangle} \dint \mu_{\SS^{n-2}} (v_1) \cdots \dint  \mu_{\SS^{n-2}} (v_n)\\
	&\qquad \le (1+\delta)^{\frac{n(n+3)}{2}}\, 2^{-n}\, z^{-n}\, f(x(u))^{n}\,  \int\limits_{\SS^{n-2}} \cdots \int\limits_{\SS^{n-2}} (\text{vol}_{n-1}([r_2(v_1) v_1,\ldots,r_2(v_n)v_n]))^2\\
	&\qquad \qquad \times \prod_{j=1}^{n} r_2(v_j)^{n-1}\, \dint \mu_{\SS^{n-2}} (v_1) \cdots \dint  \mu_{\SS^{n-2}} (v_n),
	\end{align*}   
	where again $r_2(v) = b_2(v)^{-\frac{1}{2}} z^{\frac{1}{2}}$.
Define the ellipsoid $E$ as the $(n-1)$-dimensional convex body with radial function $b_2(v)^{-\frac{1}{2}}$, i.e., $E$ is the intersection of $Q_2^{(x(u))}$ with the hyperplane $z=1$. Thus, since $\text{vol}_{n-1}$ is homogeneous, the integral appearing in the latter expression can be rewritten as an integral where the random points are chosen in the interior of $E$ according to the uniform distribution. That is, 
\begin{align*}
&\int\limits_{\SS^{n-2}} \cdots \int\limits_{\SS^{n-2}} (\text{vol}_{n-1}([r_2(v_1) v_1,\ldots,r_2(v_n)v_n]))^2\, \prod_{j=1}^{n} r_2(v_j)^{n-1}\, \dint \mu_{\SS^{n-2}} (v_1) \cdots \dint  \mu_{\SS^{n-2}} (v_n)\\
&\qquad = z^{\frac{n(n-1)}{2}}\, \int\limits_{\SS^{n-2}} \cdots \int\limits_{\SS^{n-2}} \int\limits_0^{b_2(u_1)^{- \frac{1}{2}}} \cdots \int\limits_0^{b_2(u_n)^{-\frac{1}{2}}} (\text{vol}_{n-1}([b_2(u_1)^{- \frac{1}{2}} z^{1 \over 2} u_1,\ldots,b_2(u_n)^{- \frac{1}{2}} z^{1\over 2} u_n]))^2\\
&\qquad \qquad \times \prod_{j=1}^{n} ((n-1)t_j^{n-2})\, \dint t_1 \cdots \dint t_n \dint \mu_{\SS^{n-2}}(u_1) \cdots \dint \mu_{\SS^{n-2}}(u_n)\\
&\qquad =  z^{\frac{n(n-1)}{2} + n-1}\, \int\limits_{\SS^{n-2}} \cdots \int\limits_{\SS^{n-2}} \int\limits_0^{b_2(u_1)^{- \frac{1}{2}}} \cdots \int\limits_0^{b_2(u_n)^{-\frac{1}{2}}} (\text{vol}_{n-1}([b_2(u_1)^{- \frac{1}{2}} u_1,\ldots,b_2(u_n)^{- \frac{1}{2}} u_n]))^2\\
&\qquad \qquad \times \prod_{j=1}^{n} ((n-1)t_j^{n-2})\, \dint t_1 \cdots \dint t_n \dint \mu_{\SS^{n-2}}(u_1) \cdots \dint \mu_{\SS^{n-2}}(u_n)\\
&\qquad =  z^{\frac{n^2 + n - 2}{2}}\, (n-1)^n\, \int\limits_{E} \cdots \int\limits_{E} (\text{vol}_{n-1}([\tilde{x}_1,\ldots,\tilde{x}_n]))^2 \dint x_1 \cdots \dint x_n,
\end{align*}     
where $\tilde{x}_i$ is the projection of the point $x_i$ onto the boundary of $E$, i.e., 
\begin{align*}
\tilde{x}_i = \frac{x_i}{\left\|x_i\right\|}\,  r_E\left(\frac{x_i}{\left\|x_i\right\|}\right).
\end{align*}
Here, $r_E$ is the radial function of $E$ and $\| \cdot\|$ is the Euclidean norm, with $0$ at the center of $E$.
The random elements $\dint x_i$ as well as $\text{vol}_{n-1}$ are homogeneous and invariant with respect to volume preserving affinities acting in the affine subspace $\{z=1\}$. Moreover, since the volume of $E$ equals $2^{\frac{n-1}{2}} \kappa_K(x(u))^{-\frac{1}{2}} \kappa_{n-1}$, we get by first transforming the ellipsoid $E$ into the Euclidean ball $\BB^{n-1}$ (using a suitable affinity),  then again rewriting the integral as an integral  over the sphere $\SS^{n-2}$ and finally using Lemma \ref{lem:Miles}, 
\begin{align*}
& z^{\frac{n^2 + n - 2}{2}}\, (n-1)^n\, \int\limits_{E} \cdots \int\limits_{E} (\text{vol}_{n-1}([\tilde{x}_1,\ldots,\tilde{x}_n]))^2 \dint x_1 \cdots \dint x_n\\
&\qquad =  z^{\frac{n^2 + n - 2}{2}}\, \left(2^{\frac{n-1}{2}} \kappa_K(x(u))^{-\frac{1}{2}}\right)^2 \, \left(2^{\frac{n-1}{2}} \kappa_K(x(u))^{-\frac{1}{2}}\right)^n\\
&\qquad \qquad \times \int\limits_{\SS^{n-2}} \cdots \int\limits_{\SS^{n-2}} (\text{vol}_{n-1}([x_1,\ldots,x_n]))^2 \dint \mu_{\SS^{n-2}} (x_1) \cdots \dint \mu_{\SS^{n-2}} (x_n)\\
&\qquad =  z^{\frac{n^2 + n - 2}{2}}\, 2^{\frac{n^2 + n -2}{2}}\, \kappa_K(x(u))^{-\frac{n}{2} - 1}\, \frac{n \  \omega_{n-1}^n}{(n-1)!\, (n-1)^{n-1}}\, .
\end{align*}
Combining the above calculations yields that for all sufficiently small $\delta > 0$ it holds that 
\begin{align*}
&\int\limits_{\partial K \cap H} \cdots \int\limits_{\partial K \cap H} (\text{vol}_{n-1}([x_1,\ldots,x_n]))^2 \prod_{j=1}^{n} l_{H}(x_j) \dint \PP_{f_{\partial K \cap H}}(x_1) \cdots \dint \PP_{f_{\partial K \cap H}} (x_n)\\
&\qquad \le  (1+\delta)^{\frac{n(n+3)}{2}}\, 2^{\frac{n^2 - n - 2}{2}}\,  z^{\frac{n^2  - n - 2}{2}}\,\frac{n\, \omega_{n-1}^n}{(n-1)!\, (n-1)^{n-1}}\, f(x(u))^{n}\, \kappa_K(x(u))^{-\frac{n}{2} - 1}\\
&\qquad \qquad  + \delta O(z^{\frac{n^2 - n - 2}{2}}),
\end{align*}
where the constant in $O(\cdot)$ can be chosen independently of $x(u)$ and $\delta$. This proves the lemma. 
\end{proof}     

Now, we further analyze the expression appearing in Lemma \ref{Schritt1}. In order to do this, we put $s := \PP_f(\partial K \cap H^-)$, i.e., $\PP_f(\partial K \cap H^+) = 1-s$. The results stated in \cite[equation (71)]{Reitzner} also imply the following estimates.

\begin{lem}[\cite{Reitzner}]\label{lem:Reitzner2}
	For all sufficiently small $\delta > 0$, it holds that 
	\begin{align}\label{s}
	\begin{split}
	&(1+\delta)^{-n}\, 2^{\frac{n-1}{2}}\, f(x(u))\, \kappa_K(x(u))^{- \frac{1}{2}}\, \kappa_{n-1}\, z^{\frac{n-1}{2}}\\ 
	&\qquad \le s \le (1+\delta)^{n+1}\, 2^{\frac{n-1}{2}}\, f(x(u))\, \kappa_K(x(u))^{- \frac{1}{2}}\, \kappa_{n-1}\, z^{\frac{n-1}{2}}
	\end{split}
	\end{align}
	Therefore,
	\begin{align}\label{z}
	z \le (1+\delta)^{\frac{2n}{n-1}}\, \frac{\kappa_K(x(u))^{1\over n-1}\, (n-1)^{\frac{2}{n-1}}}{2\, f(x(u))^{2\over n-1}\, \omega_{n-1}^{2\over n-1}}\, s^{2 \over n-1}
	\end{align}
	and
	\begin{align}\label{dz}
	\frac{\dint z}{\dint s} \le (1+\delta)^{n}\, \frac{\kappa_K(x(u))^{1\over 2}\, 2^{-\frac{n-3}{2}}}{f(x(u))\, \omega_{n-1}}\, z^{- \frac{n-3}{2}}.
	\end{align}
\end{lem}

Using once more results from \cite{Reitzner}, we continue the proof of the main theorem as follows.

\begin{lem}\label{Schritt2a}
For sufficiently large $N$ and sufficiently small $\delta > 0$, we have 
\begin{align*}
\EE[\text{vol}_n((1-c)K \Delta P_N)] \le I + II,
\end{align*}
where 
\begin{align*}
I &:= (1+\delta)^{\frac{3n^2 + 3n}{2}}\, a\, \binom{N}{n}\, n\, \int\limits_{\SSn}  \kappa_K(x(u))^{-1} \int\limits_{0}^{1} (1-s)^{N-n}\, s^{n-1}\, (z-ch_K(u))\, \dint s \dint \mu_{\SSn}(u)
\end{align*}
and 
\begin{align*}
II &:= (1+\delta)^{\frac{3n^2 + 3n}{2}}\, a\, \binom{N}{n}\, n\, \int\limits_{\SSn} \kappa_K(x(u))^{-1}\\
&\qquad \qquad \times \int\limits_{0}^{s(c h_K(u))} (1-s)^{N-n}\, s^{n-1}\, (ch_K(u) - z)\, \dint s \dint \mu_{\SSn}(u).
\end{align*}
Here, $z=z(s)$ and 
\begin{align*}
s(ch_K(u)) := \int\limits_{\partial K \cap H^-} f(x) \dint \mu_{\partial K}, 
\end{align*}
where $H$ is the unique hyperplane with unit normal vector $u\in \SSn$ and distance $(1-c)h_K(u)$ from the origin and $H^-$ the corresponding half space that does not contain the origin.  
\end{lem}

\begin{proof}[Proof of Lemma \ref{Schritt2a}]
Observe first that \begin{align*}
\max \{0, ((1-c)h_K(u) - h)\} = 0 \quad \text{if} \quad h > (1-c)h_K(u).
\end{align*}
This,  Lemma \ref{Schritt1},  Lemma \ref{Schritt2} and the substitution $z = h_K(u) - h$ then yield that  
\begin{align*}
&\EE[\text{vol}_n((1-c)K \Delta P_N)]\\
&\quad \le (1+\delta)^{\frac{n(n+3)}{2}}\, a\, 2^{\frac{n^2-n-2}{2}}\, \binom{N}{n}\, \frac{n\, \omega_{n-1}^n}{(n-1)^{n-1}} \int\limits_{\SSn} f(x(u))^n\, \kappa_K(x(u))^{-\frac{n}{2}-1}\\
&\qquad \qquad \times \int\limits_{h_K(u) - \epsilon}^{(1-c)h_K(u)} \left(\PP_f(\partial K \cap H^+)\right)^{N-n} z^{\frac{n^2 - n - 2}{2}} ((1-c)h_K(u) - h) \dint h \dint \mu_{\SSn}(u)\\
&\qquad + \delta \binom{N}{n} (n-1)! \int\limits_{\SSn} \int\limits_{h_K(u) - \epsilon}^{(1-c)h_K(u)} \left(\PP_f(\partial K \cap H^+)\right)^{N-n} O(z^{\frac{n^2 - n - 2}{2}})\\
&\qquad \qquad \times ((1-c)h_K(u) - h) \dint h \dint \mu_{\SSn}(u)\\
&\quad = (1+\delta)^{\frac{n(n+3)}{2}}\, a\, 2^{\frac{n^2-n-2}{2}}\, \binom{N}{n}\, \frac{n\, \omega_{n-1}^n}{(n-1)^{n-1}}\, \int\limits_{\SSn} f(x(u))^n\, \kappa_K(x(u))^{-\frac{n}{2}-1}\\
&\qquad \qquad \times \int\limits_{ch_K(u)}^{\epsilon} \left(\PP_f(\partial K \cap H^+)\right)^{N-n} z^{\frac{n^2 - n - 2}{2}} (z - c h_K(u)) \dint z \dint \mu_{\SSn}(u)\\
&\qquad+ \delta \binom{N}{n} (n-1)! \int\limits_{\SSn} \int\limits_{ch_K(u)}^{\epsilon} \left(\PP_f(\partial K \cap H^+)\right)^{N-n} O(z^{\frac{n^2 - n - 2}{2}}) (z -c h_K(u)) \dint z \dint \mu_{\SSn}(u).
\end{align*}
As the later calculations will show, the order of both summands is $N^{- \frac{2}{n-1}}$. Since $\delta$ is arbitrarily small, it is enough to consider the first summand in what follows.

We use  \eqref{dz} and then \eqref{z} to change from $z^{\frac{(n-1)^2}{2}}$ to $s^{n-1}$ and obtain that, for sufficiently large $N$, 
\begin{align*}
&\EE[\text{vol}_n((1-c)K \Delta P_N)]\\
&\quad \le (1+\delta)^{\frac{n(n+3)}{2} + n}\, a\, 2^{\frac{n^2-n-2}{2}}\, 2^{-\frac{n-3}{2}}\, \binom{N}{n}\, \frac{n\, \omega_{n-1}^n}{(n-1)^{n-1}}\, \int\limits_{\SSn} f(x(u))^{n-1}\, \kappa_K(x(u))^{-\frac{n}{2}-\frac{1}{2}}\\
&\qquad \qquad \times \int\limits_{s(ch_K(u))}^{1} (1-s)^{N-n}\, z^{\frac{n^2 - n - 2 - n + 3}{2}} (z - c h_K(u)) \dint s \dint \mu_{\SSn}(u)\\
&\quad \le (1+\delta)^{\frac{n^2 + 5n}{2}}\, a\, 2^{\frac{n^2-2n+1}{2}}\,  \binom{N}{n}\, \frac{n\, \omega_{n-1}^{n-1}}{(n-1)^{n-1}}\, \int\limits_{\SSn} f(x(u))^{n-1}\, \kappa_K(x(u))^{-\frac{n}{2}-\frac{1}{2}}\\
&\qquad \qquad \times \int\limits_{s(ch_K(u))}^{1} (1-s)^{N-n}\, z^{\frac{(n-1)^2}{2}} (z - c h_K(u)) \dint s \dint \mu_{\SSn}(u)\\
&\quad \le (1+\delta)^{\frac{n^2 + 5n}{2} + n(n-1)}\, a\, 2^{\frac{(n-1)^2}{2}}\, 2^{-\frac{(n-1)^2}{2}}\,   \binom{N}{n}\, n\, \int\limits_{\SSn}  \kappa_K(x(u))^{-1}\\
&\qquad \qquad \times \int\limits_{s(ch_K(u))}^{1} (1-s)^{N-n}\, s^{n-1}\, (z - c h_K(u)) \dint s \dint \mu_{\SSn}(u)\\
&\quad \le  (1+\delta)^{\frac{3n^2 + 3n}{2}}\, a\, \binom{N}{n}\, n\, \int\limits_{\SSn}  \kappa_K(x(u))^{-1}\\
&\qquad \qquad \times \int\limits_{s(ch_K(u))}^{1} (1-s)^{N-n}\, s^{n-1}\, (z - c h_K(u)) \dint s \dint \mu_{\SSn}(u)\\
&\quad = (1+\delta)^{\frac{3n^2 + 3n}{2}}\, a\, \binom{N}{n}\, n\, \int\limits_{\SSn}  \kappa_K(x(u))^{-1} \int\limits_{0}^{1} (1-s)^{N-n}\, s^{n-1}\, (z-ch_K(u))\, \dint s \dint \mu_{\SSn}(u)\\
&\qquad + (1+\delta)^{\frac{3n^2 + 3n}{2}}\, a\, \binom{N}{n}\, n\, \int\limits_{\SSn} \kappa_K(x(u))^{-1}\\
&\qquad \qquad \times \int\limits_{0}^{s(c h_K(u))} (1-s)^{N-n}\, s^{n-1}\, (ch_K(u) - z)\, \dint s \dint \mu_{\SSn}(u).
\end{align*} 
This proves the lemma in view of the definitions of  $I$ and $II$.
\end{proof}

We start with the first term.  

\begin{lem}\label{Schritt4}
For sufficiently large $N$ and sufficiently small $\delta > 0$, we have that
\begin{align*}
I \le  (1+\delta)^{\frac{3n^2 + 3n}{2}}\, a\, N^{- \frac{2}{n-1}}\, \int\limits_{\partial K}  \frac{\kappa_K(x)^{\frac{1}{n-1}}}{f(x)^{2\over n-1}} \dint \mu_{\partial K}(x),
\end{align*}
where $a\in (0,\infty)$ is an absolute constant. 
\end{lem}

\begin{proof}[Proof of Lemma \ref{Schritt4}]
We apply  \eqref{z} and  \eqref{AbschätzungfürC}, to get that for all sufficiently small $\delta > 0$ and sufficiently large $N$,
\begin{align*}
I & \le (1+\delta)^{\frac{3n^2 + 3n}{2}}\, a\,  \binom{N}{n}\, \frac{n}{2}\, \frac{(n-1)^{\frac{2}{n-1}}}{\omega_{n-1}^{2\over n-1}} \  \Bigg[
\  (1+\delta)^\frac{2n}{n-1} \int\limits_{\SSn}  \frac{\kappa_K(x(u))^{-1 + \frac{1}{n-1}}}{f(x(u))^{2\over n-1}} \  \dint \mu_{\SSn}(u)\\
&\qquad \qquad  \times   \int\limits_{0}^{1} (1-s)^{N-n}\, s^{n-1 + \frac{2}{n-1}}\, \dint s \\
&\qquad - \left(1-\frac{1}{n}\right)\, N^{-\frac{2}{n-1}}\, \frac{(n-1)\, \Gamma\left(n + 1 + \frac{2}{n-1}\right)}{ (n+1)!}
\frac{1}{n \text{vol}_n(K)}\, \int\limits_{\partial K} \frac{\kappa_K(x)^{1\over n-1}}{f(x)^{2\over n-1}} \dint \mu_{\partial K}(x)\\
&\qquad \qquad \times \int\limits_{\SSn} h_K(u)  \kappa_K(x(u))^{-1} \dint \mu_{\SSn}(u) \int\limits_{0}^{1} (1-s)^{N-n}\, s^{n-1}\, \dint s  \  \Bigg].
\end{align*}
For $u \in \SSn$, let $x=x(u) \in \partial K$ be such that $N_K(x)=u$. Then, the relation  $\dint \mu_{\SSn}(u) = \kappa_K(x) \dint \mu_{\partial K}(x)$ gives  that
\begin{align}\label{nvol}
n \text{vol}_n(K)= \int\limits_{\partial K} \langle x, N_K(x) \rangle \dint \mu_{\partial K}(x)=\int\limits_{\partial K} h_K(u(x)) \dint \mu_{\partial K}(x) = \int\limits_{\SSn} \frac{h_K(u) }{ \kappa_K(x(u))} \dint \mu_{\SSn}(u)
\end{align}
and that
$$
\int\limits_{\SSn}  \frac{\kappa_K(x(u))^{-1 + \frac{1}{n-1}}}{f(x(u))^{2\over n-1}}\dint \mu_{\SSn}(u) = \int\limits_{\partial K} \frac{\kappa_K(x)^{ \frac{1}{n-1}}}{f(x)^{2\over n-1}}\dint \mu_{\partial K}(x).
$$
We use those, together with the definition of the Beta function next and arrive at
\begin{align*}
I &\leq (1+\delta)^{\frac{3n^2 + 3n}{2}}\, a\,  \binom{N}{n}\, \frac{n}{2}\, \frac{(n-1)^{\frac{2}{n-1}}}{\omega_{n-1}^{2\over n-1}} \ \int\limits_{\partial K} \frac{\kappa_K(x)^{ \frac{1}{n-1}}}{f(x)^{2\over n-1}}\dint \mu_{\partial K}(x) \\ 
&\qquad \qquad \times \Bigg[ \  (1+\delta)^\frac{2n}{n-1} \frac{\Gamma(N-n+1) \Gamma\left(n+ \frac{2}{n-1}\right)}{\Gamma\left(N+1+\frac{2}{n-1}\right)} \\
&\qquad -  \left(1-\frac{1}{n}\right)\, N^{-\frac{2}{n-1}}\, \frac{(n-1) \Gamma\left(n + 1 + \frac{2}{n-1}\right)}{ (n+1)!} \  \frac{\Gamma(N-n+1) \Gamma\left(n\right)}{\Gamma\left(N+1\right)} \  \Bigg]\\
&= (1+\delta)^{\frac{3n^2 + 3n}{2}}\, a\, \frac{n}{2} \  \binom{N}{n}\, \  \frac{(n-1)^{\frac{2}{n-1}}}{\omega_{n-1}^{2\over n-1}} \ \int\limits_{\partial K} \frac{\kappa_K(x)^{ \frac{1}{n-1}}}{f(x)^{2\over n-1}}\dint \mu_{\partial K}(x) \  \frac{\Gamma(N-n+1) \Gamma\left(n+ \frac{2}{n-1}\right)}{\Gamma\left(N+1+\frac{2}{n-1}\right)}  \\
& \qquad \qquad  \times  \Bigg[ \  (1+\delta)^\frac{2n}{n-1}  - \left(1-\frac{1}{n}\right)\, N^{-\frac{2}{n-1}}\, \frac{(n-1) \left(n  + \frac{2}{n-1}\right)}{n (n+1)} \  \frac{\Gamma\left(N+1+ \frac{2}{n-1}\right) }{\Gamma\left(N+1\right)} \  \Bigg],
\end{align*}
where in the last equality  we have also used that 
\begin{align*}
\Gamma\left(n + 1 + \frac{2}{n-1}\right) = \Gamma\left(n + \frac{2}{n-1}\right) \left(n + \frac{2}{n-1}\right).
\end{align*}
Now, observe  that  for sufficiently large $N$, 
\begin{align*}
\frac{\Gamma(N-n+1) \Gamma\left(n+ \frac{2}{n-1}\right)}{\Gamma\left(N+1+\frac{2}{n-1}\right)} \sim \frac{1}{\binom{N}{n} n N^{2\over n-1}} \quad \text{and } \quad \Gamma(N+1+\frac{2}{n-1}) \sim N^{2\over n-1} \Gamma\left(N+1\right).
\end{align*}
Thus, we obtain that  for sufficiently large $N$ and sufficiently small $\delta$,
\begin{align*}
 I &\leq (1+\delta)^{\frac{3n^2 + 3n}{2}}\, a\, N^{-\frac{2}{n-1}} \, \frac{(n-1)^{\frac{2}{n-1}}}{\omega_{n-1}^{2\over n-1}} \ \int\limits_{\partial K} \frac{\kappa_K(x)^{ \frac{1}{n-1}}}{f(x)^{2\over n-1}}\dint \mu_{\partial K}(x)  \\
&\qquad \times \Bigg[ \  (1+\delta)^\frac{2n}{n-1} - \left(1-\frac{1}{n}\right) \frac{ (n-1)\left(n  + \frac{2}{n-1}\right)}{n (n+1)} \ \Bigg]\\
&\leq (1+\delta)^{\frac{3n^2 + 3n}{2}}\, a\, N^{-\frac{2}{n-1}} \, \int\limits_{\partial K} \frac{\kappa_K(x)^{ \frac{1}{n-1}}}{f(x)^{2\over n-1}}\dint \mu_{\partial K}(x), 
\end{align*}
where in the last inequality we have also used that $\omega_{n-1}^{2\over n-1} \sim \frac{1}{n}$ and that $(n-1)^{2\over n-1} \le 2$.
\end{proof}

Now, we deal with the second summand in Lemma \ref{Schritt2a}. 
\newpage
\begin{lem}\label{Schritt5}
For sufficiently large $N$ and sufficiently small $\delta > 0$, we have that 
$$
II \leq    (1+\delta)^{\frac{3n^2 + 3n}{2}} a\, N^{-\frac{2}{n-1}} \, \int\limits_{\partial K} \frac{\kappa_K(x)^{ \frac{1}{n-1}}}{f(x)^{2\over n-1}}\dint \mu_{\partial K}(x),
$$
where $a\in (0,\infty)$ is an absolute constant.
\end{lem}
  
\begin{proof}[Proof of Lemma \ref{Schritt5}]	
By definition of $II$, 
\begin{align*}
II &= (1+\delta)^{\frac{3n^2 + 3n}{2}}\, a\, \binom{N}{n}\, n\, \int\limits_{\SSn} \kappa_K(x(u))^{-1}\\
&\qquad \times \int\limits_{0}^{s(c h_K(u))} (1-s)^{N-n}\, s^{n-1}\, (ch_K(u) - z)\, \dint s \dint \mu_{\SSn}(u),
\end{align*}
where $a\in (0,\infty)$ is an absolute constant.
First of all, notice that by (\ref{c}) we get that
\begin{equation}\label{nochmal:c}
c \le a\, \frac{1}{\vol_n(K)\, N^{\frac{2}{n-1}} } \  \int\limits_{\partial K} \frac{\kappa_K(x)^{ \frac{1}{n-1}}}{f(x)^{2\over n-1}}\dint \mu_{\partial K}(x).
\end{equation}
By Lemma \ref{lem:Reitzner2} and again by (\ref{c}),
\begin{align}\label{boundschKu}
\begin{split}
s(c h_K(u)) &\leq (1 + \delta )^{n+1} \  2^{\frac{n-1}{2}} \  \kappa_{n-1}  \  \frac{ f (x(u)) \  h_K(u)^{\frac{n-1}{2}}}{ \kappa_K(x(u))^\frac{1}{2}} \  c^{\frac{n-1}{2}}\\
&\leq \frac{(1 + \delta )^{n+1}\,e^{\frac{1}{12}}\, n}{e \ N}   \ \frac{ f (x(u)) \  h_K(u)^{\frac{n-1}{2}}}{ \kappa_K(x(u))^\frac{1}{2}}    \left(\frac{\int\limits_{\partial K} \frac{\kappa_K(x)^{ \frac{1}{n-1}}}{f(x)^{2\over n-1}}\dint \mu_{\partial K}(x)}{n \ \vol_n(K)}  \right) ^{\frac{n-1}{2}},
\end{split}
\end{align}
since it has been shown in \cite{LudwigSchuettWerner} that
\begin{align*}
\left(\frac{n-1}{n+1}\right)^{\frac{2}{n-1}} \sim \frac{1}{e}\quad \text{and}\quad  \left(\frac{\Gamma\left(n+1+\frac{2}{n-1}\right)}{n!}\right)^{\frac{n-1}{2}} \le e^{\frac{1}{12}} n.
\end{align*}
Now, we distinguish two cases.\\
\underline{Case 1:}   \hskip 6mm  $s(c h_K(u)) \leq \frac{(n-1)^\frac{n-1}{n} }{ n^\frac{1}{2(n-1)} \  N }$.\\
The function $(1-s)^{N-n} s^{n-1}$ has its maximum at $ \frac{n-1}{N-1}$ and $\frac{(n-1)^\frac{n-1}{n}}{  n^\frac{1}{2(n-1)} \  N } \leq \frac{n-1}{N-1}$. 
Therefore, $(1-s)^{N-n} s^{n-1}$ is increasing on $\left[0, \frac{(n-1)^\frac{n-1}{n}}{n^\frac{1}{2(n-1)} \  N }\right]$ and thus, since $  \binom{N}{n}\, n\,  \sim \frac{N^n e^n}{\sqrt{2\pi} \sqrt{n} \    n^{n-1}}$,
\begin{align*} 
&\binom{N}{n}\, n\,   \int\limits_{0}^{s(c h_K(u))} (1-s)^{N-n}\, s^{n-1}\, (ch_K(u) - z)\, \dint s\\
&\qquad \le c\, h_K(u) \ \binom{N}{n} n\,   \int\limits_{0}^{s(c h_K(u))} (1-s)^{N-n}\, s^{n-1}  \dint s\\
&\qquad \le c\, h_K(u) \ \binom{N}{n} n\,  s(c h_K(u))  \left(1-\frac{(n-1)^\frac{n-1}{n}}{n^\frac{1}{2(n-1)} \  N }\right)^{N-n}\, \left(\frac{(n-1)^\frac{n-1}{n}}{n^\frac{1}{2(n-1)} \  N }\right)^{n-1}\\
&\qquad \le a\, c\, h_K(u) \frac{N^n e^n}{\sqrt{2\pi} \sqrt{n} \   n^{n-1}}  \frac{(n-1)^\frac{n-1}{n} }{ n^\frac{1}{2(n-1)} \  N } \left(1-\frac{(n-1)^\frac{n-1}{n}}{n^\frac{1}{2(n-1)} \  N }\right)^{N-n}\, \left(\frac{(n-1)^\frac{n-1}{n}}{n^\frac{1}{2(n-1)} \  N }\right)^{n-1}\\
&\qquad \le a\,  c\, h_K(u)\, \frac{e^n}{n}\, \exp\left(-\frac{(N-n)(n-1)^\frac{n-1}{n}}{n^\frac{1}{2(n-1)} \  N }\right)\\
&\qquad \le a \  \frac{ c\, h_K(u)}{n}, 
\end{align*}
where we used in the last step that $\frac{(n-1)^{\frac{n-1}{n}}}{n^{\frac{1}{2(n-1)}}} \sim n$.
Hence, with (\ref{nochmal:c}) and \eqref{nvol}, 
\begin{align*}
II &\leq (1+\delta)^{\frac{3n^2 + 3n}{2}}\ a \  \frac{c}{n}   \int\limits_{\SSn} \kappa_K(x(u))^{-1}\  h_K(u) \dint \mu_{\SSn}(u)\\
&\leq (1+\delta)^{\frac{3n^2 + 3n}{2}}\ a \  N^{-\frac{2}{n-1}} \ \int\limits_{\partial K} \frac{\kappa_K(x)^{ \frac{1}{n-1}}}{f(x)^{2\over n-1}}\dint \mu_{\partial K}(x) , 
\end{align*}
which finishes the proof of the lemma in Case 1. 
\vskip 2mm
\noindent
\underline{Case 2:}   \hskip 6mm  $s(c h_K(u)) > \frac{(n-1)^{\frac{n-1}{n}}}{ n^\frac{1}{2(n-1)} \  N }$.\\
The inequality $s(c h_K(u)) > \frac{(n-1)^{\frac{n-1}{n}}}{ n^\frac{1}{2(n-1)} \  N }$ is in view of \eqref{boundschKu} equivalent to 
$$
(1+\delta)^{n+1}\,   \   \frac{n\, f (x(u)) \  h_K(u)^{\frac{n-1}{2}}}{e^{\frac{11}{12}}\, \kappa_K(x(u))^\frac{1}{2}}    \left(\frac{\int\limits_{\partial K} \frac{\kappa_K(x)^{ \frac{1}{n-1}}}{f(x)^{2\over n-1}}\dint \mu_{\partial K}(x)}{n \ \vol_n(K)}  \right) ^{\frac{n-1}{2}} > \   \frac{(n-1)^\frac{n-1}{n}}{  n^\frac{1}{2(n-1)}  },
$$
which is equivalent to 
$$
   h_K(u)   \   \frac{\int\limits_{\partial K} \frac{\kappa_K(x)^{ \frac{1}{n-1}}}{f(x)^{2\over n-1}}\dint \mu_{\partial K}(x)}{n \ \vol_n(K)}  > \   
   \frac{e^\frac{11}{6(n-1)} \ (n-1) ^\frac{2}{n}} { (1+\delta)^\frac{2(n+1)}{n-1} \   n^\frac{2n-1}{(n-1)^2} }  \  \frac{\kappa_K(x(u))^{ \frac{1}{n-1}}}{f(x(u))^{2\over n-1}}.
$$
We integrate both sides  over $\partial K$ with respect to $\mu_{\partial K}$ and get that
$$
(1+\delta)^\frac{2(n+1)}{n-1}  \  >  \   \frac{e^\frac{11}{6(n-1)} \ (n-1) ^\frac{2}{n}} {n^\frac{2n-1}{ (n-1)^2} } ,
$$
and arrive at a contradiction, since for all $n\ge 3$ the right hand side is strictly bigger than $1$ and $\delta$ can be chosen arbitrarily small. This shows that Case 2 is not possible and therefore finishes the proof of the lemma.
\end{proof}

\begin{proof}[Proof of Theorem \ref{mainresult}]
Lemma \ref{Schritt4} and Lemma \ref{Schritt5} imply that for sufficiently large $N$ and sufficiently small $\delta > 0$, 
\begin{align}\label{last}
\EE[\text{vol}_n((1-c)K \Delta P_N)] \le (1+\delta)^{\frac{3n^2 + 3n}{2}}\, a\, N^{- \frac{2}{n-1}} \int\limits_{\partial K}  \frac{\kappa_K(x)^{\frac{1}{n-1}}}{f(x)^{2\over n-1}} \dint \mu_{\partial K}(x),
\end{align}
where $a\in (0,\infty)$ is an absolute constant. Taking into account that we were approximating the body $(1-c)K$ instead of $K$, we need to multiply the bound  \eqref{last} by $(1-c)^{-n}$. Since
\begin{align*}
(1-c)^n \ge 1 - nc,
\end{align*}
for sufficiently large $N$,  we have  that $(1-c)^{-n} \le a$, where $a\in (0,\infty)$ is an absolute constant.  
Finally, since the bound \eqref{last} holds for all $\delta > 0$, the theorem follows.
\end{proof}

\end{document}